\documentclass{amsart}
\usepackage{mathrsfs}
\usepackage{amssymb}
\usepackage{latexsym}
\usepackage{yfonts}
\usepackage{natbib}
\usepackage{graphicx}

\bibpunct{[}{]}{;}{n}{,}{,}

\newtheorem{te}{Theorem}[section]

\theoremstyle{lem}

\theoremstyle{coro}

\theoremstyle{definition}

\theoremstyle{os}
\newtheorem{os}[te]{Remark}

\numberwithin{equation}{section}

\allowdisplaybreaks

\begin{document}

\title[Higher-order equations]{Probabilistic representation of fundamental solutions to $\frac{\partial u}{\partial t} = \kappa_m \frac{\partial^m u}{\partial x^m}$.}

\author{Enzo Orsingher} 
\address{Dipartimento of Scienze Statistiche, Sapienza University of Rome}
\email{enzo.orsingher@uniroma1.it}

\author{Mirko D'Ovidio}
\address{Dipartimento di Scienze di Base e Applicate per l'Ingegneria, Sapienza University of Rome}
\email{mirko.dovidio@uniroma1.it}
 

\date{\today}

\subjclass[2000]{Primary }

\begin{abstract}
For the fundamental solutions of heat-type equations of order $n$ we give a general stochastic representation in terms of damped oscillations with generalized gamma distributed parameters. By composing the pseudo-process $X_n$ related to the higher-order heat-type equation with positively  skewed stable r.v.'s $T^j_\frac{1}{3}$, $j=1,2, \ldots , n$ we obtain genuine r.v.'s whose explicit distribution is given for $n=3$ in terms of Cauchy asymmetric laws. We also prove that $X_3(T^1_\frac{1}{3}(\ldots (T^n_\frac{1}{3}(t))\ldots))$ has a stable asymmetric law.
\end{abstract}

\maketitle

\section{Introduction}
The problem of studying the form of fundamental solutions of higher-order heat equations of the form
\begin{equation}
\frac{\partial u_m}{\partial t}(x,t) = \kappa_{m} \frac{\partial^m u_m}{\partial x^m}(x,t), \quad x \in \mathbb{R}, \; t>0, \quad m \geq 2 \label{eqNord}
\end{equation}
with
\begin{equation*}
u_m(x,0)=\delta(x)
\end{equation*}
where
\begin{equation*}
\kappa_{m} = \left\lbrace \begin{array}{ll} (-1)^{m/2 + 1} & \textrm{if } $m$ \textrm{ is even}\\ \pm 1 & \textrm{if } $m$ \textrm{ is odd}\end{array} \right .
\end{equation*}
has been tackled in some particular cases by mathematicians of the caliber of \citet{Bernst19, lev23, pol39} and \citet{burwell23}. By applying the steepest descent method some recent papers by \citet{liWong1993}, \citet{AccettaO}, \citet{LCH03} have explored the form of the fundamental solutions of equation \eqref{eqNord}. The aim of this note is to give an explicit representation of the solutions to \eqref{eqNord} for the case where the order of the equation is odd, alternative to the inverse Fourier transform 
\begin{equation}
u_m(x,t) = \frac{1}{2\pi} \int_{\mathbb{R}} e^{-i \beta x + t (-i \beta)^m \kappa_m} d\beta
\end{equation} 
and capable of representing the sign-varying behavior of the fundamental solutions to \eqref{eqNord}. Our result is that the fundamental solutions to \eqref{eqNord} have the probabilistic representation
\begin{equation}
u_{2n+1}(x,t) = \frac{1}{x \pi} E\left[ e^{-b_nxG^{2n+1}(1/t)} \sin\left( a_nxG^{2n+1}(1/t) \right) \right] , \quad x\in \mathbb{R},\; t>0 \label{solIntro1}
\end{equation}
in the odd-order case, and
\begin{equation}
u_{2n}(x,t) = \frac{1}{\pi x} E \Big[ \sin\left(x G^{2n}(1/t) \right) \Big] , \quad x \in \mathbb{R}, \; t>0\label{solIntro2}
\end{equation}
for the even-order case. In \eqref{solIntro1} and \eqref{solIntro2} by $G^\gamma(t)$ we denote the generalized  gamma r.v. with density 
\begin{equation*}
g^{\gamma}(x,t) = \gamma \frac{x^{\gamma -1}}{t}\exp\left(- \frac{x^\gamma}{t} \right), \quad x,t>0, \; \gamma >0.
\end{equation*}
The parameters $a_n,b_n$ appearing in \eqref{solIntro1} and \eqref{solIntro2} are
\begin{align*}
a_n=\cos \frac{\pi}{2(2n+1)}, \quad b_n = \sin \frac{\pi}{2(2n+1)}.
\end{align*}
Results \eqref{solIntro1} and \eqref{solIntro2} show that the fundamental solutions have an oscillating behavior which has been explored in several papers by many researchers. In our view our result represents a concluding picture of the solutions to higher-order heat equations. For all values of the degree $n$ of the equation \eqref{eqNord} we have solutions  which have the behavior of damped oscillations where the probabilistic ingredients (the generalized gamma or Weibull-type distributions) depend only on $n \in \mathbb{N}$. An alternative universal representation of the fundamental solution in the odd-order case reads
\begin{equation}
u_{2n+1}(x,t) = -\frac{1}{\pi x} \sum_{k=1}^{\infty} \left( \frac{-x}{\sqrt[2n+1]{t}} \right)^{k} \frac{1}{k!} \sin \frac{n \pi k}{2n+1} \Gamma\left( 1+ \frac{k}{2n+1} \right).
\end{equation} 
Functions $u_{2n+1}$ display ascillations which fade off as the degree $2n+1$ of the equation increases. A special attention has been devoted to third-order equations where we have that
\begin{align}
u_3(x,t) = & \frac{1}{\sqrt[3]{3t}} Ai\left( \frac{x}{\sqrt[3]{3t}} \right), \quad x \in \mathbb{R},\, t>0\\
= & \frac{3t}{\pi x} \int_{0}^\infty e^{-\frac{xy}{2}} \sin \left(\frac{\sqrt{3}}{2}xy\right)\, y^2 e^{-t y^3}\, dy\\
= & - \frac{1}{\pi x} \sum_{k=1}^{\infty} \left( -\frac{x}{\sqrt[3]{t}} \right)^k \frac{1}{k!} \sin \frac{\pi k}{3} \Gamma\left( 1+ \frac{k}{3} \right).
\end{align}
In the fourth-order case (biquadratic heat-equation) in \citet{DO3} we have shown that
\begin{align*}
u_4(x,t) = & \frac{1}{2\pi} \int_{-\infty}^{+\infty} e^{-\frac{y^4 t}{2^2}} \cos\left(xy \right)dy \\
= & \frac{1}{2\pi \sqrt{2t^{1/2}}} \sum_{k=0}^{\infty} \frac{(-1)^k }{(2k)!} \left( - \frac{\sqrt{2}|x|}{t^{1/4}} \right)^{2k} \Gamma\left( \frac{k}{2} + \frac{1}{4} \right). 
\end{align*}

In a recent paper we have shown that the composition of an odd-order pseudo-process $X_{2n+1}$ with a positivly skewed stable r.v. $T_\frac{1}{2n+1}$ of order $\frac{1}{2n+1}$ yields a genuine r.v. with asymmetric Cauchy distribution, that is 
\begin{equation}
Pr\{ X_{\frac{1}{2n+1}}(T_{\frac{1}{2n+1}}(t)) \in dx \}= \frac{t\cos \frac{\pi}{2(2n+1)}}{\pi \left[  (x + t\sin \frac{\pi}{2 (2n+1)})^2 + t^2 \cos^2 \frac{\pi}{2(2n+1)}\right]}dx. \label{lab110}
\end{equation}
For $n=1$ from \eqref{lab110} we can extract a very interesting relationship for the Airy function which reads
\begin{align*}
& Pr\{ X_{\frac{1}{3}}(T_{\frac{1}{3}}(t)) \in dx \} /dx= \\
& \int_0^\infty \frac{1}{\sqrt[3]{3s}} Ai\left( \frac{x}{\sqrt[3]{3s}}\right) \frac{t}{s} \frac{1}{\sqrt[3]{3s}}Ai\left( \frac{t}{\sqrt[3]{3s}}\right)ds =  \frac{\sqrt{3}}{2\pi} \frac{t}{x^2 + xt + t^2}.
\end{align*}
We show here that the $m$-times iterated pseudo-process (with $T^j_{\frac{1}{2n+1}}$, $j=1,2, \ldots , m$ independent stable r.v.'s)
\begin{equation*}
Z_{m}(t) = X_{2n+1}(T^1_{\frac{1}{2n+1}}(T^2_{\frac{1}{2n+1}}(\ldots (T^m_{\frac{1}{2n+1}}(t)) \ldots))), \quad t>0
\end{equation*}
is a stable r.v. of order $\frac{1}{(2n+1)^{m-1}}$ with characteristic function
\begin{equation}
Ee^{i\beta Z_{m}(t)} = \exp\left[ -t |\beta |^{\frac{1}{(2n+1)^{m-1}}} \left( \cos \frac{\pi}{2(2n+1)^m} + i \frac{\beta}{|\beta |} \sin \frac{\pi}{2(2n+1)^m} \right) \right]. \label{wwwintro}
\end{equation}
We have also explored the connection between solutions of fractional equations
\begin{equation*}
\frac{\partial^\alpha u}{\partial t^\alpha} + \frac{\partial u}{\partial x} = 0, \quad x>0,\; t>0
\end{equation*}
with the solutions of higher-order heat-type equations \eqref{eqNord} for $\alpha = \frac{1}{m}$, $m \in \mathbb{N}$.

\section{Pseudo-processes}
Some basic facts about the fundamental solutions of higher-order heat equations had been established many years ago essentially by applying the steepest descent method. In particular, \citet{liWong1993} have shown that the number of zeros is infinite for solutions to even-order equations. The steepest descent method was applied by \citet{AccettaO} for the analysis of the third-order equation. The oscillating behavior of the solutions of higher-order heat-type equations is confirmed by our analysis. Furthermore, for the odd-order case our results show that the asymmetry of solutions decreases as the order $2n+1$ increases. The result of Theorem \ref{te21} below shows that solutions of all odd-order heat equations can be constructed by means of damped oscillating functions  with gamma distributed parameters.

We pass now to our principal result.
\begin{te}
The solution to
\begin{equation}
\left\lbrace 
\begin{array}{ll}
\frac{\partial u}{\partial t} = (-1)^n \frac{\partial^{2n+1} u}{\partial x^{2n+1}}, & x\in \mathbb{R},\; t>0\\
u(x,0) = \delta(x)
\end{array} \right .
\label{eqN}
\end{equation}
is given by 
\begin{align}
u_{2n+1}(x,t) = \frac{1}{x \pi} E\left[ e^{-b_nxG^{2n+1}(1/t)} \sin\left( a_nxG^{2n+1}(1/t) \right) \right]  \label{soleqN}
\end{align}
where, $\forall \, t>0$, the r.v. $G^{\gamma}(t)$ has the generalized gamma distribution
\begin{equation*}
g^{\gamma}(x,t) = \gamma \frac{x^{\gamma -1}}{t}\exp\left(- \frac{x^\gamma}{t} \right), \quad x,t>0, \; \gamma >0.
\end{equation*}
and
\begin{align*}
a_n=\cos \frac{\pi}{2(2n+1)}, \quad b_n = \sin \frac{\pi}{2(2n+1)}
\end{align*}
\label{te21}
\end{te}
\begin{proof}
We start by evaluating the Fourier transform of \eqref{soleqN}
\begin{align}
& \int_{\mathbb{R}} e^{i \beta x} u_{2n+1}(x,t) dx\label{ftNord}\\
= & \int_{\mathbb{R}} e^{i \beta x} dx \int_0^\infty \frac{(2n+1) t}{\pi x} e^{-b_n x w} \sin(a_n xw) w^{2n} e^{-t w^{2n+1}} dw\nonumber \\
= & (2n+1)t \int_0^\infty w^{2n} e^{-t w^{2n+1}} \int_{\mathbb{R}} \frac{e^{i \beta x + i a_n x w- b_n xw} - e^{i \beta x -i a_n xw - b_n xw}}{2\pi ix}dx\, dw\nonumber \\
= & (2n+1) t \int_0^\infty w^{2n} e^{-t w^{2n+1}}\left[ H_{\beta}(w(a_n-ib_n)) - H_{\beta}(-w(a_n+ib_n)) \right] dw \nonumber
\end{align}
where in the last step we used the integral representation of the Heaviside function
\[ H_{y}(x) = - \frac{1}{2\pi} \int_{\mathbb{R}} e^{-iwx} \frac{e^{iyw}}{iw} dw = \frac{1}{2\pi} \int_{\mathbb{R}} e^{iwx} \frac{e^{-iyw}}{iw} dw \]
By a change of variable, the Fourier transform \eqref{ftNord} takes the form
\begin{align}
\int_{\mathbb{R}} e^{i \beta x} u_{2n+1}(x,t) dx = & \frac{(2n+1)t}{(a_n-ib_n)^{2n+1}} \int_0^\infty w^{2n} e^{-t \left(\frac{w}{a_n-ib_n} \right)^{2n+1}}H_{\beta}(w) dw\notag \\
& - \frac{(2n+1)t}{(a_n+ib_n)^{2n+1}} \int_0^\infty w^{2n} e^{-t \left( \frac{w}{a_n+ib_n} \right)^{2n+1}} H_{\beta}(-w)dw\notag \\
= & i (2n+1)t \int_{0}^\infty w^{2n} e^{-i tw^{2n+1} }H_{\beta}(w)dw\notag \\
& + i (2n+1)t  \int_0^\infty w^{2n} e^{it w^{2n+1}} H_{\beta}(-w) dw\notag \\
= & i (2n+1)t \int_{-\infty}^{+\infty} w^{2n} e^{-i tw^{2n+1} }H_{\beta}(w)dw\notag \\
= & i (2n+1)t \int_{\beta}^{+\infty} w^{2n} e^{-i tw^{2n+1} } dw. \label{intabove}
\end{align}
In the above steps we used the fact that
\begin{equation*}
(a_n+ib_n)^{2n+1}=i, \quad \textrm{and} \quad (a_n-ib_n)^{2n+1}=-i.
\end{equation*}
The integral \eqref{intabove} can be performed in two different ways. First we can take the Laplace transform 
\begin{align*}
\int_0^\infty e^{-\mu t} \left( \int_{\mathbb{R}} e^{i \beta x} u_{2n+1}(x, t) dx \right) dt = & \int_{\beta}^{\infty}  w^{2n} \int_0^\infty e^{-(\mu + i w^{2n+1})t} it (2n+1)dt\, dw\\
= & \int_{\beta}^\infty \frac{i (2n+1) w^{2n} dw}{(\mu + i w^{2n+1})^2}\\
= & \frac{1}{\mu + i \beta^{2n+1}} \\
= & \int_0^\infty e^{-\mu t} e^{-i t \beta^{2n+1}}dt.
\end{align*}
This shows that
\begin{align*}
\int_{-\infty}^{+\infty} e^{i\beta x} u_{2n+1}(x, t)dx = e^{-i t \beta^{2n+1}}.
\end{align*}
We can arrive at the some result by means of the following trick
\begin{align*}
\int_{-\infty}^{+\infty} e^{i\beta x} u_{2n+1}(x, t)dx = & \lim_{\mu \to 0} \left[ i (2n+1)t \int_{\beta}^{\infty} w^{2n} e^{-i tw^{2n+1} - \mu w^{2n+1} }dw \right] \\
= & \lim_{\mu \to 0 } \frac{it}{\mu + it} e^{- (i t  + \mu) \beta^{2n+1}}\\
= & e^{-i t \beta^{2n+1}}.
\end{align*}
We have thus shown that the Fourier transform of \eqref{soleqN} coincides with the Fourier transform of the solution to the Cauchy problem \eqref{eqN}.
\end{proof}

For the special case of the third-order heat equation we have the following result.
\begin{te}
The solution of the Cauchy problem
\begin{equation}
\left\lbrace \begin{array}{ll} \frac{\partial u}{\partial t} = - \frac{\partial^3 u}{\partial x^3}, \quad x \in \mathbb{R},\, t>0\\
u(x,0) = \delta(x) \end{array} \right .
\end{equation}
can be written as
\begin{align}
u_3(x,t) = & \frac{1}{x \pi} E\left[ e^{-\frac{x}{2}G^{3}(1/t)} \sin\left( \frac{3^{1/2} x}{2}G^{3}(1/t) \right) \right]\\
= & \frac{3t}{\pi x} \int_{0}^\infty e^{-\frac{xy}{2}} \sin \left(\frac{\sqrt{3}}{2}xy\right)\, y^2 e^{-t y^3}\, dy\\
= & \frac{1}{\sqrt[3]{3t}} Ai\left( \frac{x}{\sqrt[3]{3t}} \right), \quad x \in \mathbb{R},\, t>0.
\end{align}
\end{te}
\begin{proof}
It is convenient to work with the following series expansion of the Airy function (see \citet[formula (4.10)]{OB09})
\begin{align}
Ai(w) = & \frac{3^{-2/3}}{\pi} \sum_{k \geq 0} (3^{1/3} w)^k \frac{\sin\left(2\pi (k+1)/3 \right)}{k!} \Gamma\left( \frac{k+1}{3} \right)\label{uno8tto}\\
= & \frac{3^{-2/3}}{\pi} \sum_{k \geq 1} (3^{1/3} w)^{k-1} \frac{\sin\left(2\pi k/3 \right)}{(k-1)!} \Gamma\left( \frac{k}{3} \right)\nonumber \\
= & \frac{1}{w \pi} \sum_{k \geq 1} (3^{1/3} w)^{k} \frac{\sin\left(2\pi k/3 \right)}{k!} \Gamma\left( \frac{k}{3} +1 \right)\nonumber.
\end{align}
If we expand the function
\begin{align}
g(x, \phi) = & e^{x \cos \phi} \sin\left( x \sin \phi \right) = e^{x \cos \phi} \frac{e^{ix\sin \phi} - e^{-i x \sin \phi}}{2i} \label{uno11}\\
= & \frac{e^{xe^{i\phi}} - e^{x e^{-i\phi}}}{2i} = \frac{1}{2i} \left[ \sum_{k=0}^{\infty} \frac{\left( xe^{i\phi} \right)^k}{k!} - \sum_{k=0}^{\infty} \frac{\left( xe^{-i\phi} \right)^k}{k!} \right] \nonumber \\
= & \sum_{k=0}^{\infty} \frac{x^k}{k!} \frac{e^{i\phi k} - e^{-i\phi k}}{2i} = \sum_{k=0}^{\infty} \frac{x^k}{k!} \sin \phi k \nonumber
\end{align}
we establish a relationship which is useful in transforming \eqref{uno8tto} as
\begin{align*}
Ai(w) = & \frac{1}{w \pi} \sum_{k \geq 1} (3^{1/3} w)^{k} \frac{\sin\left(2\pi k/3 \right)}{k!} \Gamma\left( \frac{k}{3} +1 \right)\\
= & \frac{1}{w \pi} \sum_{k \geq 1} (3^{1/3} w)^{k} \frac{\sin\left(2\pi k/3 \right)}{k!} \int_0^\infty z^{k/3} e^{-z} dz\\
= & \frac{1}{w \pi} \int_0^\infty e^{-z}  \sum_{k \geq 1} ((3z)^{1/3} w)^{k} \frac{\sin\left(2\pi k/3 \right)}{k!}\, dz\\
= & \frac{1}{w \pi} \int_0^\infty e^{-z} e^{- (3z)^{1/3} w/2} \sin\left( \frac{3^{5/6}}{2} z^{1/3} w \right)\, dz.
\end{align*}
Now we write
\begin{align}
\frac{1}{(3t)^{1/3}}Ai\left( \frac{x}{(3t)^{1/3}} \right) = & \frac{3t}{x \pi} \int_0^\infty z^2 \exp\left( -z^3 t - \frac{z x}{2} \right) \sin \left( \frac{3^{1/2}}{2} z\, x \right) \, dz.
\end{align}
From \eqref{soleqN}, for $n=1$ ($\gamma=3$), we write
\begin{align*}
u_3(x, t) = & \frac{3t}{\pi x} \int_0^\infty  e^{- xz\sin \frac{\pi}{6}} \sin \left( xz \cos \frac{\pi}{6} \right) \, z^2 e^{-t z^3}\, dz\\
= & (z^3t = y) =  \frac{1}{\pi x} \int_0^\infty e^{-y} e^{- \frac{x}{2}\sqrt[3]{\frac{y}{t}}} \sin \left( \frac{\sqrt{3}}{2} x \sqrt[3]{\frac{y}{t}}\right) dy.
\end{align*}
This proves, in a different way, that
\begin{equation*}
u_3(x, t) = \frac{1}{\sqrt[3]{3t}}Ai\left( \frac{x}{\sqrt[3]{3t}} \right).
\end{equation*}
\end{proof}

\begin{te}
We can write the fundamental solution $u_{2n+1}$ in the following alternative form
\begin{equation}
u_{2n+1}(x,t) = -\frac{1}{\pi x} \sum_{k=1}^{\infty} \left( \frac{-x}{\sqrt[2n+1]{t}} \right)^{k} \frac{1}{k!} \sin \frac{n \pi k}{2n+1} \Gamma\left( 1+ \frac{k}{2n+1} \right)
\end{equation}
\end{te}
\begin{proof}
From \eqref{soleqN} we have that
\begin{align*}
u_{2n+1}(x, t) = & \frac{(2n+1)t}{\pi x} \int_0^\infty e^{-xy \sin \frac{\pi}{2(2n+1)}} \sin \left( xy\cos\frac{\pi}{2(2n+1)} \right) y^{2n} e^{-t y^{2n+1}} dy\\
= & \frac{(2n+1)t}{\pi x} \int_0^\infty e^{-xy \cos \frac{n\pi}{(2n+1)}} \sin \left( xy\sin\frac{n\pi}{(2n+1)} \right) y^{2n} e^{-t y^{2n+1}} dy\\
= & (\textrm{by } \eqref{uno11} ) =  -\frac{(2n+1)t}{\pi x} \int_{0}^\infty y^{2n} e^{-t y^{2n+1}} \sum_{k=0}^{\infty} \frac{(-xy)^k}{k!} \sin \frac{n\pi k}{2n+1} dy\\
= & - \frac{1}{\pi x} \sum_{k=0}^\infty \left( \frac{-x}{t^{\frac{1}{2n+1}}} \right)^k \frac{1}{k!} \sin\left(\frac{n \pi k}{2n+1} \right) \Gamma\left(1 + \frac{k}{2n+1} \right)
\end{align*}
\end{proof}

\begin{os}
\normalfont
We note that
\[ u_{2n+1}(0,t) = \frac{1}{\pi t^{\frac{1}{2n+1}}} \sin \left( \frac{n\pi}{2n+1} \right) \Gamma\left( 1 + \frac{1}{2n+1}\right) \stackrel{n \to \infty}{\longrightarrow} \frac{1}{\pi} \]
\end{os}

\begin{os}
\normalfont
We are able to evaluate the semi-infinite integral 
$$ \int_0^\infty u_{2n+1}(x,t)\, dx $$
by means of the representation \eqref{soleqN}. Indeed, by considering that
\begin{align}
\int_{0}^{\infty} \frac{e^{-Bx}}{x} \sin\left( Ax \right) \, dx = & \operatorname{arctan} \left( \frac{A}{B} \right) \label{arctab} \\
= & \frac{\pi}{2} - \operatorname{arccot} \left( \frac{A}{B} \right) \notag
\end{align}
(see \citet[formula 3.941]{GR}) we get that
\begin{align}
\int_0^\infty u_{2n+1}(x,t)\, dx = & \int_0^\infty \frac{1}{x \pi} E\left[ e^{-bxG^{2n+1}(1/t)} \sin\left( axG^{2n+1}(1/t) \right) \right] \, dx\nonumber \\
= & \frac{1}{\pi} \left[ \frac{\pi}{2} - \operatorname{arccot} \left( \frac{a}{b} \right) \right]\nonumber \\
= & \frac{1}{\pi} \left[ \frac{\pi}{2} - \operatorname{arccot} \left( \frac{\cos \frac{\pi}{2} \frac{1}{2n+1}}{\sin \frac{\pi}{2} \frac{1}{2n+1}} \right) \right]\nonumber \\
= & \frac{1}{2} \left[ 1 -   \frac{1}{2n+1} \right]. \label{accordsInt}
\end{align}
which is in accord with Lachal \cite[formula 11]{LCH03}.
\end{os}

\begin{te}
The solution to
\begin{equation}
\frac{\partial u}{\partial t} = (-1)^{n+1} \frac{\partial^{2n}u}{\partial x^{2n}}, \quad x \in \mathbb{R}, \; t>0 \label{sdfR}
\end{equation}
with initial condition $u(x,0)=\delta(x)$ can be written as
\begin{equation}
u_{2n}(x,t) = \frac{1}{\pi x} E \Big[ \sin\left(x G^{2n}(1/t) \right) \Big] \label{solu2n}
\end{equation}
\end{te}
\begin{proof}
The solution to \eqref{sdfR} is given by  
$$u_{2n}(x,t) = \frac{1}{\pi} \int_{0}^\infty e^{-t \beta^{2n}} \cos \beta x\, d\beta.$$
By integrating by parts we get that
\begin{equation*}
u_{2n}(x,t) = \frac{\sin \beta x}{\pi x} e^{-t \beta^{2n}} \Bigg|_{0}^{\infty} + \frac{2n t}{\pi x} \int_0^\infty \beta^{2n-1} e^{-t \beta^{2n}} \sin \beta x\, d\beta
\end{equation*}
and this concludes the proof.
\end{proof}

\begin{os}
\normalfont
The solution \eqref{solu2n} takes the following values in the origin
\begin{equation*}
u_{2n}(0,t) = \frac{1}{\pi} E G^{2n}(1/t) = \frac{1}{\pi t^{1/2n}} \Gamma\left( 1 + \frac{1}{2n} \right) \stackrel{n \to \infty}{\longrightarrow} \frac{1}{\pi}.
\end{equation*}
Clearly, in force of the symmetry and by considering formula \eqref{arctab}, we obtain that
$$ \int_{0}^{\infty} u_{2n}(x,t) dx = \frac{1}{2}.$$
\end{os}

\begin{os}
\normalfont
It is well-known that the solution to the fractional diffusion equation
\begin{equation}
\left\lbrace \begin{array}{l} \frac{\partial^{\nu} u}{\partial t^{\nu}} = \lambda^2 \frac{\partial^2 u}{\partial x^2}, \quad x \in \mathbb{R},\; t>0, \nu \in (0,2]\\ u(x,0)=\delta(x) \end{array} \right .
\label{eqWri}
\end{equation}
with 
$$u_t(x, 0)=0 \quad \textrm{for} \quad \nu \in (1,2]$$
is given by
\begin{align*}
u_{\nu}(x,t) = & \frac{1}{2\lambda t^{\nu/2}} W_{-\frac{\nu}{2} ,1-\frac{\nu}{2} } \left(- \frac{|x|}{\lambda t^{\nu/2}} \right)\\
= & \frac{1}{2\lambda t^{\nu/2}} \sum_{k=0}^{\infty} \frac{1}{k!} \left(-\frac{|x|}{\lambda t^{\nu/2}} \right) \frac{1}{\Gamma\left( 1 - \frac{\nu}{2}(1+k) \right)}\\
= & \frac{1}{2\pi \lambda t^{\nu/2}} \sum_{k=0}^{\infty} \frac{1}{k!} \left(-\frac{|x|}{\lambda t^{\nu/2}} \right) \Gamma\left( \frac{\nu}{2}(1+k)\right) \sin \left( \frac{\pi \nu}{2}(1+k)\right).
\end{align*}
The folded solution to the equation \eqref{eqWri} reads
\begin{equation*}
\bar{u}_{\nu}(x,t) = \left\lbrace \begin{array}{ll} 2u_{\nu}(x,t), & x\geq 0\\ 0, & x<0\end{array}  \right .
\end{equation*}
and for $\nu=2\alpha$, $\alpha \in (0,1)$, $\lambda=1$ becomes
\begin{align}
q_{\alpha}(x, t) = & \frac{1}{\pi t^{\alpha}} \sum_{r=0}^{\infty} \left(-\frac{x}{t^{\alpha}}\right)^r \frac{\Gamma(\alpha(r+1))}{r!} \sin\left(\pi \alpha(r+1) \right). \label{rvX}
\end{align} 
This represents a probability density of a r.v. $X(t)$ on the half-line $(0,\infty)$ which can be expressed in terms of positively skewed stable densities.
\begin{align*}
Pr\{ X(t) \in dx \} = & q_{\alpha}(x,t) dx = \left( \frac{x}{t^\alpha} = \frac{t}{y^{\alpha}} \right) = q_{\alpha}\left( \frac{t^{\alpha +1}}{y^{\alpha +1}} \right) \frac{\alpha t^{\alpha +1}}{y^{\alpha +1}}dy\\
= & \frac{1}{\pi t^{\alpha}} \sum_{k=0}^{\infty} \frac{1}{k!} \left( - \frac{t}{y^{\alpha}} \right)^k \Gamma(\alpha (1+k)) \sin \left( \pi \alpha (1+k) \right) \frac{\alpha t^{\alpha +1}}{y^{\alpha +1}}dy\\
=& \frac{\alpha t}{\pi y^{\alpha +1}} \sum_{k=0}^{\infty} \frac{1}{k!} \left(- \frac{t}{y^{\alpha}} \right)^{k} \Gamma(\alpha (1+k)) \sin \left( \pi \alpha (1+k) \right)dy\\
= & \frac{\alpha}{\pi} \sum_{k=0}^{\infty} \frac{1}{k!} \Gamma(\alpha (1+k)) y^{-\alpha(k+1)-1} t^{k+1} \sin\left( \pi \alpha (1+k)\right) dy\\
= & p_{\alpha}(y, \alpha, t) \, dy
\end{align*}
where in the last step the expression of the stable density
\begin{equation}
p_{\alpha}(x, \alpha , 1) = \frac{\alpha}{\pi} \sum_{k=0}^{\infty} \frac{1}{k!} \Gamma(\alpha(1+k)) x^{-\alpha (1+k)-1}\sin (\pi \alpha (1+k)), \quad x \in \mathbb{R}_+. \label{serieslawS1}
\end{equation}
The calculations above show that the r.v. $X(t)$ with distribution \eqref{rvX} can be expressed as
$$X(t) = t \left( \frac{t}{Y(t)}\right)^\alpha$$
where $Y(t)$ is a positively skewed stable-distributed r.v. of order $\alpha \in (0,1)$. In other words the stable law of $Y(t)$ is related to the folded solution of the fractional diffusion equation $X(t)$ in the sense that
$$ Y(t) \stackrel{i.d.}{=} t\left( \frac{t}{X(t)} \right)^{1/\alpha}.$$
This is because
\begin{equation*}
Pr\left\lbrace  Y(t) < y \right\rbrace  = Pr\left\lbrace  X(t) > \frac{t^{\alpha +1}}{y^{\alpha}} \right\rbrace  = \int_{\frac{t^{\alpha +1}}{y^{\alpha}}}^{\infty} p_{\alpha}(x,\alpha,t)dx.
\end{equation*}
We give also the Laplace transforms with respect to time $t$ and space $x$ of $q_{\alpha}(x,t)$ 
\begin{align}
\int_0^\infty e^{-\lambda x} q_{\alpha}(x,t) dx = & - \left[ - \frac{1}{\lambda t^{\alpha}} E_{-\alpha, 1-\alpha}\left( - \frac{1}{\lambda t^{\alpha}} \right) \right] \label{tlapISS}\\
= & E_{\alpha, 1}(- \lambda t^{\alpha}), \quad t>0 \notag
\end{align}
where in the last step, formula 
\begin{equation*}
-\frac{1}{x}E_{-\alpha, 1-\alpha}\left(\frac{1}{x} \right) = E_{\alpha, 1}(x)
\end{equation*}
has been applied (see formula $(5.1)$ of \citet{OB09EJP}). Furthermore, 
\begin{equation}
\int_0^\infty e^{-\mu t} q_{\alpha}(x,t) dt = \mu^{\alpha -1} e^{-x \mu^\alpha}, \quad x>0.
\end{equation}
Formulas above help to check that $q_{\alpha}$ satisfies the fractional equation
\begin{equation}
\frac{\partial^{\alpha} u}{\partial t^{\alpha}} + \frac{\partial u}{\partial x} = 0 \label{eqISS}
\end{equation}
with iniital condition $u(x,0)=\delta(x)$ where $\frac{\partial^\alpha}{\partial t^\alpha}$ is the Caputo fractional derivative.
\end{os}
\begin{os}
\normalfont
The double Laplace transform of \eqref{eqISS} reads
\begin{equation}
\mu^\alpha \mathcal{L}(\lambda, \mu) - \mu^{\alpha -1} = \mathcal{L}(\lambda, \mu)
\end{equation}
where
$$ \mathcal{L}(\lambda, \mu) = \int_0^\infty e^{-\mu t} \int_0^\infty e^{-\lambda x} q_{\alpha}(x, t)\, dx\, dt.$$
Since
\begin{equation}
\mathcal{L}(\lambda, \mu)= \frac{\mu^{\alpha -1}}{\mu^\alpha + \lambda} \label{xtlapISS}
\end{equation}
the $t$-inverse Laplace transform is \eqref{tlapISS} while the $x$-inverse Laplace transform can be obtained from \eqref{xtlapISS} as
\begin{equation}
\frac{\mu^{\alpha -1}}{\mu^{\alpha} + \lambda} = \mu^{\alpha -1} \int_0^\infty e^{-x(\lambda + \mu^{\alpha})} dx.
\end{equation}
\end{os}

\begin{os}
\normalfont
For $\alpha=1/2$, the formula \eqref{serieslawS1} yields
\begin{align}
p_{1/2}(x;1/2, 1) = & \frac{1}{2\pi} \sum_{r=0}^{\infty} (-1)^r \frac{\Gamma((r+1)/2)}{r!} x^{-\frac{(r+1)}{2}-1} \sin\left(\pi (r+1)/2 \right)\nonumber \\
= & \frac{1}{2\pi} \sum_{k=0}^{\infty}\left[  \frac{\Gamma\left( k + \frac{1}{2} \right)}{(2k)!} x^{-k -\frac{3}{2}} \sin\left(\pi k +  \frac{\pi}{2} \right) - \frac{\Gamma(k)}{(2k-1)!} x^{-k-1} \sin\left(\pi k \right)  \right]\nonumber \\
= & \frac{1}{2\pi x^{3/2}} \sum_{k=0}^{\infty}  \frac{\Gamma\left( k + \frac{1}{2} \right)}{(2k)!} x^{-k} \sin\left(\pi k + \frac{\pi}{2} \right)\nonumber \\
= & \frac{1}{2\pi x^{3/2}} \sum_{k=0}^{\infty}(-1)^k  \frac{\Gamma\left( k + \frac{1}{2} \right)}{(2k)!} x^{-k}\nonumber \\
= & \frac{1}{2\pi x^{3/2}} \sum_{k=0}^{\infty}\frac{(-1)^k}{(2k)!} 2^{1-2k} \frac{\Gamma(2k)}{\Gamma(k)} \sqrt{\pi} x^{-k}\nonumber \\
= & \frac{1}{2 \sqrt{\pi} x^{3/2}} \sum_{k=0}^{\infty} \frac{(-1)^k}{k!}  2^{-2k} x^{-k}\nonumber \\
= & \frac{1}{\sqrt{4 \pi x^3}} \exp\left(- \frac{1}{4x} \right) \label{firstpass}.
\end{align}
The result \eqref{firstpass} shows that the stable law \eqref{serieslawS1} for $\alpha=1/2$ coincides with the first-passage time of a standard Brownian motion through level $1/\sqrt{2}$.
\end{os}

\begin{os}
\normalfont
Two solutions to the third-order p.d.e
\begin{equation}
\frac{\partial u}{\partial t} = - \frac{\partial^3 u}{\partial x^3} \label{eq3ord}
\end{equation}
are given by
\begin{align*}
p(x,t) = & \frac{1}{\sqrt[3]{3t}} Ai\left( \frac{x}{\sqrt[3]{3t}}\right)\end{align*}
and
\begin{align*}
q(x,t) = & \frac{x}{t} \frac{1}{\sqrt[3]{3t}} Ai\left( \frac{x}{\sqrt[3]{3t}}\right) = \frac{x}{t}p(x,t).
\end{align*}
Indeed, we have that
\begin{align*}
\frac{\partial q}{\partial t}(x,t) = &- \frac{x}{t^2} p(x, t) + \frac{x}{t} \frac{\partial p}{\partial t}(x,t)
\end{align*}
and
\begin{align*}
\frac{\partial q}{\partial x}(x, t) = & \frac{1}{t}p(x, t) + \frac{x}{t} \frac{\partial p}{\partial x}(x,t),\\
\frac{\partial^2 q}{\partial x^2}(x, t) = & \frac{2}{t} \frac{\partial p}{\partial x}(x, t) + \frac{x}{t} \frac{\partial^2 p}{\partial x^2}(x, t),\\
\frac{\partial^3 q}{\partial x^3}(x, t) = & \frac{3}{t} \frac{\partial^2 p}{\partial x^2}(x,t) + \frac{x}{t}\frac{\partial^3 p}{\partial x^3}(x, t), \\
= & \frac{3}{t} \frac{\partial^2 p}{\partial x^2}(x,t) - \frac{x}{t}\frac{\partial p}{\partial t}(x, t)
\end{align*}
and therefore
\begin{equation*}
\frac{\partial q}{\partial t}(x,t) + \frac{\partial^3 q}{\partial x^3}(x, t) = \frac{1}{t^2} \left( 3t \frac{\partial^2 p}{\partial x^2}(x,t) - x p(x, t) \right).
\end{equation*}
By observing that
$$ \frac{\partial^2}{\partial x^2}\left[ \frac{1}{\sqrt[3]{3t}} y\left(\frac{x}{\sqrt[3]{3t}}\right) \right] = \frac{1}{3t}y^{\prime \prime}(z) \Bigg|_{z=\frac{x}{\sqrt[3]{3t}}} = \frac{1}{3t} y^{\prime \prime}\left( \frac{x}{\sqrt[3]{3t}} \right) $$
and the fact that  
$$ y^{\prime \prime}(z) - z y(z)=0 $$
that is, $y$ satisfies the Airy equation, we get that
$$\frac{\partial q}{\partial t}(x,t) + \frac{\partial^3 q}{\partial x^3}(x, t) =0.$$
By recursive arguments it can be shown that
\begin{equation}
f_m(x, t) = \left( \frac{x}{t} \right)^m \frac{1}{\sqrt[3]{3t}} Ai\left( \frac{x}{\sqrt[3]{3t}} \right)
\end{equation}
is also a solution to \eqref{eq3ord} for $m>1$.
\end{os}
\begin{os}
\normalfont
We have shown in a previous paper that the r.v.
\begin{equation}
Z(t) = X_3(T_\frac{1}{3}(t)) \label{eq228}
\end{equation}
obtained by composing the third-order pseudo-process $X_3$ with the stable subordinator $T_\frac{1}{3}$ with distribution
\begin{equation}
Pr\{ T_\frac{1}{3}(t) \in ds \} = ds \frac{t}{s} \frac{1}{\sqrt[3]{3s}} Ai\left( \frac{t}{\sqrt[3]{3s}} \right), \quad s,t>0 \label{eq229}
\end{equation}
possesses Cauchy distribution
\begin{align}
Pr\{Z(t) \in dx \} = & dx \int_0^\infty  \frac{1}{\sqrt[3]{3s}} Ai\left( \frac{x}{\sqrt[3]{3s}} \right) \frac{t}{s} \frac{1}{\sqrt[3]{3s}} Ai\left( \frac{t}{\sqrt[3]{3s}} \right)ds \notag \\
= & dx \frac{\sqrt{3} t}{2\pi (x^2+ xt + t^2)} \notag \\
= & dx \frac{\sqrt{3}}{2\pi}\frac{t}{(x+\frac{t}{2})^2 + \frac{3}{4}t^2}.\label{eq230}
\end{align}
Result \eqref{eq230} shows that \eqref{eq228} is a genuine r.v..The characteristic function of \eqref{eq230} is clearly 
\begin{equation}
\int_{-\infty}^{+\infty} e^{i\beta x} Pr\{ Z(t) \in dx \} = e^{- \frac{\sqrt{3}}{2} t |\beta | - i \frac{t}{2}\beta}.
\end{equation}
\end{os}

We have now the following generalization of the previous result for the composition of the pseudo-process $X_3$ with successively composed subordinators of order $\frac{1}{3}$.
\begin{te}
The r.v.
\begin{equation}
Z_n(t) = X_3(T^1_\frac{1}{3}(\ldots (T^n_\frac{1}{3}(t)) \ldots ))
\end{equation}
with $T^j_\frac{1}{3}$, $j=1,2, \ldots , n$, independent, positively skewed r.v.'s with law \eqref{eq229} has characteristic function
\begin{align}
Ee^{i \beta Z_n(t)} = & \exp \left[ -t |\beta |^{\frac{1}{3^{n-1}}} \left( \cos \frac{\pi}{2\cdot 3^n} + i \, sgn (\beta) \, \sin \frac{\pi}{2\cdot 3^n} \right)\right] \notag \\
= & \exp \left[- t |\beta |^{\frac{1}{3^{n-1}}} \cos \frac{\pi}{2\cdot 3^n} \left( 1 + i \, sgn (\beta) \, \tan \frac{\pi}{2\cdot 3^n} \right) \right] \label{eq233}
\end{align}
\end{te}
\begin{proof}
We first observe that 
\begin{align*}
& Pr\{ Z_n(t) \in dx \}/dx \\
= & \int_0^\infty \ldots \int_0^\infty \frac{1}{\sqrt[3]{3s_1}} Ai\left( \frac{x}{\sqrt[3]{3s_1}} \right) \frac{s_2}{s_1}\frac{1}{\sqrt[3]{3s_1}} Ai\left( \frac{s_2}{\sqrt[3]{3s_1}} \right) \ldots \frac{t}{s_n}\frac{1}{\sqrt[3]{3s_n}} Ai\left( \frac{t}{\sqrt[3]{3s_n}} \right) \prod_{j=1}^n ds_j
\end{align*}
has Fourier transform
\begin{align*}
& \int_{-\infty}^{+\infty} e^{i\beta x} Pr\{ Z_n(t) \in dx \}\\
= & \int_0^\infty \ldots \int_0^\infty e^{-i \beta^3 s_1}\;  \frac{s_2}{s_1}\frac{1}{\sqrt[3]{3s_1}} Ai\left( \frac{s_2}{\sqrt[3]{3s_1}} \right) \ldots \frac{t}{s_n}\frac{1}{\sqrt[3]{3s_n}} Ai\left( \frac{t}{\sqrt[3]{3s_n}} \right) \prod_{j=1}^n ds_j\\
= & \int_0^\infty \ldots \int_0^\infty e^{- (i \beta^3)^\frac{1}{3} s_2}\;  \frac{s_3}{s_2}\frac{1}{\sqrt[3]{3s_2}} Ai\left( \frac{s_3}{\sqrt[3]{3s_2}} \right) \ldots \frac{t}{s_n}\frac{1}{\sqrt[3]{3s_n}} Ai\left( \frac{t}{\sqrt[3]{3s_n}} \right) \prod_{j=2}^n ds_j\\
= & \exp \left[ - t\left( e^{i \frac{\pi}{2}} \, sgn(\beta) \, |\beta |^3 \right)^{\frac{1}{3^n}}  \right]\\
= & \exp \left[ - t\left( e^{i \frac{\pi}{2}\, sgn(\beta) }  |\beta |^3 \right)^{\frac{1}{3^n}}  \right]\\
= & \exp\left[ -t |\beta |^{\frac{1}{3^{n-1}}} \left( \cos \frac{\pi}{2\cdot 3^n} + i \, sgn (\beta)\, \sin \frac{\pi}{2\cdot 3^n}   \right) \right].
\end{align*}
We observe that the characteristic function of a stable r.v. $S_\alpha(t)$ can be written as
\begin{align}
E e^{i \beta S_{\alpha}(t)} = & e^{-t |\beta |^\alpha e^{-i \frac{\pi \nu}{2}}\frac{\beta}{|\beta |}}\\
= & \exp\left[ -\sigma t |\beta |^\alpha \left( 1 - i \theta \frac{\beta }{|\beta |} \tan \frac{\pi \alpha }{2} \right) \right] \notag
\end{align}
where 
\begin{equation*}
\theta = \frac{\tan \frac{\pi \nu}{2}}{\tan \frac{\pi \alpha}{2}} \in [-1, 1]
\end{equation*}
and $\sigma=\cos \frac{\pi \nu}{2}>0$. In our case $\alpha = \frac{1}{3^{n-1}}$, $\nu=\frac{1}{3^n}$ and therefore $\sigma= \cos \frac{\pi}{2\cdot 3^n}$ and
\begin{equation}
\theta = \frac{\tan \frac{\pi}{2\cdot 3^n}}{\tan \frac{\pi }{2\cdot 3^{n-1}}}\label{asymT}
\end{equation}
and since $\theta \neq \pm 1$, the r.v. $Z_n$ is spread on the whole line with parameter of asymmetry equal to \eqref{asymT}.
\end{proof}

\begin{os}
\normalfont
We note that by adjusting the derivation of \eqref{eq233} we can obtain result \eqref{wwwintro} of the introduction.  
\end{os}

\begin{os}
\normalfont
The positively skewed stable r.v. $T_\alpha(t)$, $t>0$, $\alpha \in (0,1)$, with Laplace transform
\begin{equation*}
E e^{-\lambda T_\alpha(t)} = e^{-\lambda^\alpha t}
\end{equation*}
has characteristic function
\begin{align*}
E e^{i \beta T_\alpha(t)} =& E e^{- (-i \beta) T_\alpha(t)}\\
= & e^{- t (-i \beta)^\alpha }\\
= & \exp \left[-t |\beta |^\alpha \,  e^{-i \frac{\pi \alpha}{2} \, sgn(\beta)} \right]\\
= & \exp \left[-t |\beta |^\alpha \left( \cos \frac{\pi \alpha}{2} - i \, sgn(\beta) \, \sin \frac{\pi \alpha}{2}\right) \right]\\
= & \exp \left[- t |\beta |^\alpha   \cos \frac{\pi \alpha}{2} \left(1 - i \frac{\beta}{|\beta |} \tan \frac{\pi \alpha}{2}\right) \right]
\end{align*}
and therefore with asymmetric parameter $\theta=+1$ and $\sigma= \cos \frac{\pi \alpha }{2}$.
\end{os}

\begin{figure}[ht]
\centering
\includegraphics[scale=.7]{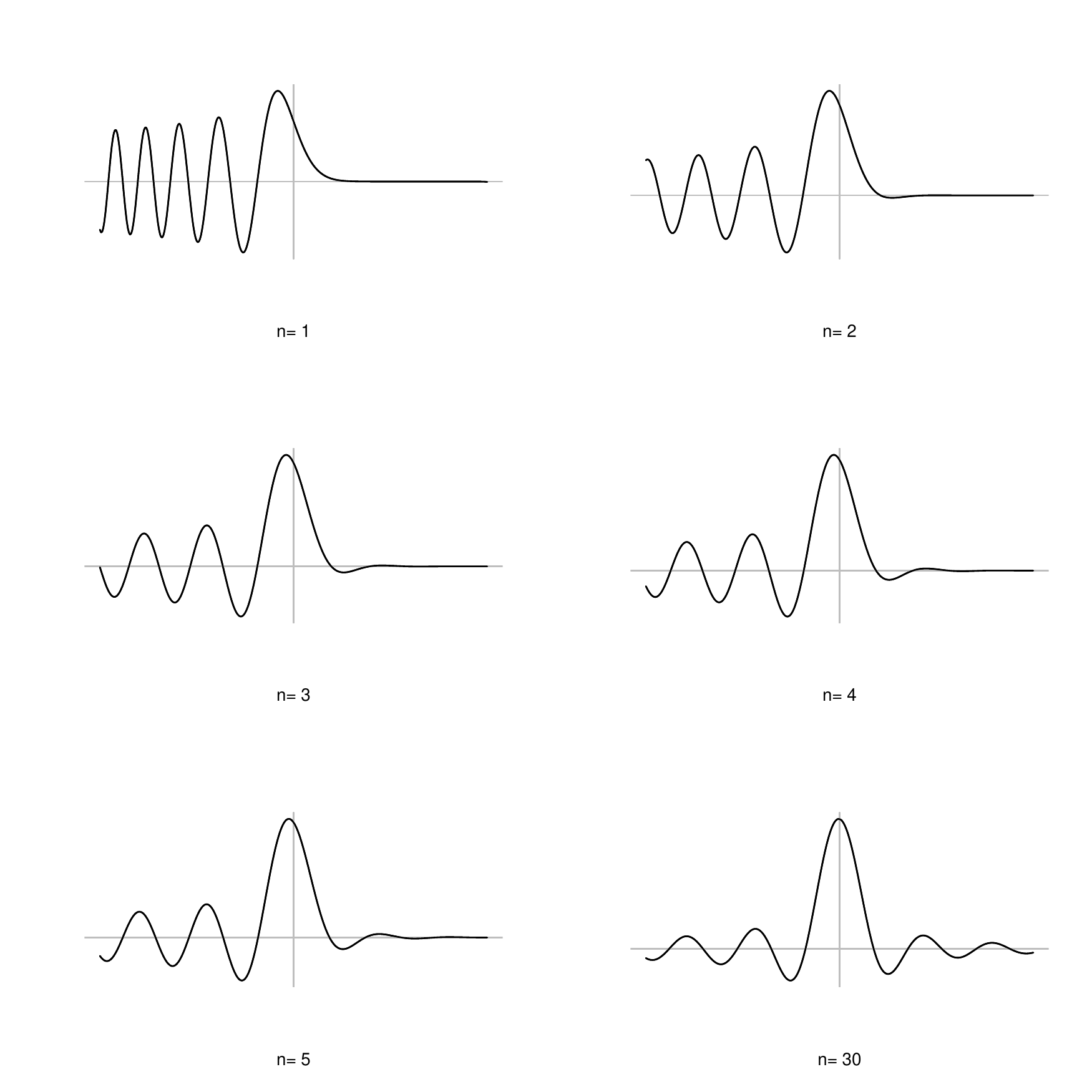} 
\caption{Profiles of odd-order solutions, $m=2n+1$}
\end{figure}

\end{document}